\documentclass[11pt]{article}
\usepackage{graphicx}
\usepackage{amsthm, amsmath, amssymb, tikz, bm}
\usepackage{mathtools, intcalc}
\usepackage{xifthen}
\usepackage[colorlinks=true,
linkcolor=blue,citecolor=blue,
urlcolor=blue]{hyperref}

\usepackage[margin=1.2in]{geometry} 

\usepackage[shortlabels]{enumitem}
\usepackage{todonotes}
\usetikzlibrary{calc,shapes, backgrounds}
\allowdisplaybreaks

\newtheorem{lemma}{Lemma}
\newtheorem{theorem}{Theorem}

\newtheorem{definition}{Definition}
\newtheorem{claim}{Claim}

\newcommand{\ds}{\displaystyle}
\newcommand{\dss}{\displaystyle\sum}

\newcommand{\lp}{\left(}
\newcommand{\rp}{\right)}

\usepackage[normalem]{ulem}

\pgfdeclarelayer{edgelayer}
\pgfdeclarelayer{nodelayer}
\pgfsetlayers{edgelayer,nodelayer,main}

\title{On the size of planar graphs with positive Lin-Lu-Yau Ricci curvature}

\author{
Linyuan Lu
\thanks{University of South Carolina, Columbia, SC 29208,
({\tt lu@math.sc.edu}). This author was supported in part by NSF
grant DMS-1600811.} \and
Zhiyu Wang \thanks{Georgia Institute of Technology, Atlanta, GA, 30332, (zwang672@gatech.edu).} 
}

\begin{document}
\maketitle

\begin{abstract}
We show that if a planar graph $G$ with minimum degree at least $3$ has positive Lin-Lu-Yau Ricci curvature on every edge, then $\Delta(G)\leq 17$, which then implies that $G$ is finite. This is an analogue of a result of DeVos and Mohar [{\em Trans. Amer. Math. Soc., 2007}] on the size of planar graphs with positive combinatorial curvature.

\end{abstract}

\section{Introduction}\label{sec:ricci}

Let $G$ be a simple connected planar graph that is $2$-cell embedded in the sphere, and let $V, E, F$ be the set of vertices, edges, and faces of $G$. Given $v \in V$, let $\deg(v)$ denote the number of edges containing $v$, and let $F(v)$ denote the set of faces that touch $v$. Given a face $\sigma \in F$, let $|\sigma|$ denote the \emph{size} of $\sigma$, which is the number of edges bounding $\sigma$.
For each $v\in V$, the \emph{combinatorial curvature} at $v$, denoted by $\phi(v)$, is defined by $\phi(v) = 1 - \frac{\deg(v)}{2} + \sum_{\sigma \in F(v)} \frac{1}{|\sigma|}$. A graph $G$ is said to have positive combinatorial curvature (or \textit{positively curved}) if $\phi(v) > 0$ for all $v\in V$.
Higuchi \cite{Higuchi01} conjectured that that if $G$ is a simple connected positively curved graph embedded into a $2$-sphere and $\delta(G)\geq 3$, then $G$ is finite. Higuchi's conjecture was verified by Sun and Yu \cite{Sun-Yu04} for cubic planar graphs and resolved by DeVos and Mohar \cite{DeVos-Mohar07}. In particular, DeVos and Mohar showed the following theorem.

\begin{theorem}[\cite{DeVos-Mohar07}]\label{thm:DM}
Suppose $G$ is a connected simple graph embedded into a $2$-dimensional topological manifold $\Omega$ without boundary and $G$ has minimum degree at least $3$. If $G$ has positive combinatorial curvature, then it is finite and $\Omega$ is homeomorphic to either a $2$-sphere or a projective plane. Moreover, if $G$ is not a prism, an antiprism, or one of their projective plane analogues, then $|V(G)| \leq 3444$.
\end{theorem}

The  minimum possible constants for $|V(G)|$ in Theorem \ref{thm:DM} for $G$ embedded in a $2$-sphere and projective plane respectively was studied in \cite{RBK05, Nicholson-Sneddon11, Chen-Chen08, Zhang08, Oh17}. In particular, Nicholson and Sneddon \cite{Nicholson-Sneddon11} gave examples of positively curved graphs with $208$ vertices embedded into a $2$-sphere. The upper bound on $|V(G)|$ was recently settled by Ghidelli \cite{Ghidelli17}.

In this paper, we study the analogous version of the Higuchi's conjecture in the context of the Lin-Lu-Yau Ricci curvature (which will be defined later). 
The definition of the (non-combinatorial) Ricci curvature on metric spaces first came from the Bakry and Emery notation \cite{Bakry-Emery} who defined the ``lower Ricci curvature bound" through the heat semigroup $(P_t)_{t\geq 0}$ on a metric measure space. Ollivier \cite{Ollivier} defined the coarse Ricci curvature of metric spaces in terms of how much small balls are closer (in Wasserstein transportation distance) then their centers are. This notion of coarse Ricci curvature on discrete spaces was also made explicit in the Ph.D. thesis of Sammer \cite{Sammer}.
Under the assumption of positive curvature in a metric space, Gaussian-like or Poisson-like concentration inequalities can be obtained. Such concentration inequalities have been investigated in \cite{Joulin} for time-continuous Markov jump processes and in \cite{Ollivier, JO} in metric spaces.
The first definition of Ricci curvature on graphs was introduced by Chung and Yau in \cite{Chung-Yau95}. For a more general definition of Ricci curvature, Lin and Yau \cite{LY} gave a generalization of the lower Ricci curvature bound in the framework of graphs. In \cite{LLY}, Lin, Lu, and Yau modified Ollivier's Ricci curvature \cite{Ollivier} and defined a new variant of Ricci curvature on graphs, which does not depend on the idleness of the random walk. 

In this paper, we adopt the same notation as in \cite{LLY}.
Given a vertex $v\in V(G)$, we use $\Gamma(v)$ to denote the neighborhood of $v$, i.e., $\Gamma(v) = \{u: vu\in E(G)\}$. More generally, given a subset of vertices $S\subseteq V(G)$, $\Gamma(S) = \{u \in V(G): us\in E(G) \textrm{ for some } s \in S\}$. Given $x,y\in V(G)$, define $\Gamma(x,y) = \{u: u\in \Gamma(x)\cap \Gamma(y)\}$, i.e., the common neighbors of $x$ and $y$. 

A probability distribution (over the vertex set $V(G)$) is a mapping $m: V\to [0,1]$ satisfying $\sum_{x\in V} m(x) = 1$. Suppose two probability distributions $m_1$ and $m_2$ have finite support. A {\em coupling} between $m_1$ and $m_2$ is a mapping $A: V\times V \to [0,1]$ with finite support so that 
$$\dss_{y \in V} A(x,y) = m_1(x) \textrm{ and } \dss_{x\in V} A(x,y) = m_2(y).$$
Let $d(x,y)$ be the graph distance between two vertices $x$ and $y$. The \textit{transportation distance} between two probability distributions $m_1$ and $m_2$ is defined as follows:
$$W(m_1, m_2) = \inf_A \dss_{x,y\in V} A(x,y) d(x,y),$$
where the infimum is taken over all coupling $A$ between $m_1$ and $m_2$. By the duality theorem of a linear optimization problem, the transportation distance can also be written as follows:
$$W(m_1, m_2) = \sup_f \dss_{x\in V} f(x) \lp m_1(x)-m_2(x)\rp,$$
where the supremum is taken over all $1$-Lipschitz functions $f$.

A random walk $m$ on $G=(V,E)$ is defined as a family of probability measures $\{m_v(\cdot)\}_{v\in V}$ such that $m_v(u) = 0$ for all $uv \notin E$. It follows that  $m_v(u) \geq 0$ for all $v,u\in V$ and $\sum_{u\in \Gamma(v)} m_v(u) = 1$. The Ricci cuvature $\kappa: \binom{V(G)}{2} \to \mathbb{R}$ of $G$ can then be defined as follows:

\begin{definition}
Given $G=(V,E)$, a random walk $m = \{m_v(\cdot)\}_{v\in V}$ on $G$ and two vertices $x,y\in V$,
$$\kappa(x,y) = 1 - \frac{W(m_x, m_y)}{d(x,y)}.$$
Moreover, we say a graph $G$ equipped with a random walk $m$ has Ricci curvature at least $\kappa_0$ if $\kappa(x,y) \geq \kappa_0$ for all $x,y \in V$.
\end{definition}

For $0\leq \alpha < 1$, the {\em $\alpha$-lazy random walk} $m_x^{\alpha}$ (for any vertex $x$), is defined as 
\[
m_x^{\alpha}(v) = \begin{cases} 
                        \alpha & \textrm{ if $v=x$,}\\
                        (1-\alpha)/d(x) &\textrm{ if $v\in \Gamma(x)$,}\\
                        0 & \textrm{ otherwise.}
                    \end{cases}
\]
In \cite{LLY}, Lin, Lu and Yao defined the Ricci curvature of graphs based on the $\alpha$-lazy random walk as $\alpha$ goes to $1$. More precisely,
for any $x,y \in V$, they defined the $\alpha$-Ricci-curvature $\kappa_{\alpha}(x,y)$ to be 
$$\kappa_{\alpha}(x,y) = 1 - \frac{W(m_x^{\alpha}, m_y^{\alpha})}{d(x,y)}$$ and the Ricci curvaure $\kappa_{\textrm{LLY}}$ of $G$ to be 
\[\kappa_{\textrm{LLY}}(x,y) = \ds\lim_{\alpha \to 1} \frac{\kappa_{\alpha}(x,y)}{(1-\alpha)}.\]
They showed in \cite{LLY} that $\kappa_{\alpha}$ is concave in $\alpha \in [0,1]$ for any two vertices $x,y$. Moreover, 
\[\kappa_{\alpha}(x,y) \leq (1-\alpha) \frac{2}{d(x,y)}.\]
for any $\alpha \in [0,1]$ and any two vertices $x$ and $y$. In particular, this implies the following lemma.

\begin{lemma}\cite{LLY, Ollivier}\label{lem:diam}
If for any edge $xy$, $\kappa_{\textrm{LLY}}(x,y) \geq \kappa_0 > 0$, the the diameter of the graph $G$
$$\textrm{diam}(G) \leq \frac{2}{\kappa_0}.$$
\end{lemma}

Although the Ricci curvature $\kappa_{\textrm{LLY}}(x,y)$ is defined for all pairs $x,y \in V(G)$, it suffices to consider  only $\kappa_{\textrm{LLY}}(x,y)$ for $xy\in E(G)$ due to the following lemma:

\begin{lemma}\cite{LLY, Ollivier}\label{lem:adj-pair}
If $\kappa_{\textrm{LLY}}(x,y) \geq \kappa_0$ for any edge $xy\in E(G)$, then $\kappa_{\textrm{LLY}}(x,y) \geq \kappa_0$ for any pair of vertices $\{x,y\}$.
\end{lemma}

M\"{u}nch and Wojciechowski \cite{MW} gave a limit-free formulation of the Lin-Lu-Yau Ricci curvature using \textit{graph Laplacian}. Given a graph $G = (V,E)$, the combinatorial graph Laplacian $\Delta$ is defined as: 
$$\Delta f(x)=\frac{1}{d(x)}\sum\limits_{y\in \Gamma(x)} (f(y)-f(x)). $$

\begin{theorem}\label{thm:curvature_laplacian}\cite{MW} (Curvature via the Laplacian) Let $G= (V,E)$ be a simple graph and let $x \neq y \in V(G)$. Then 

\begin{equation*}
    \kappa_{\textrm{LLY}}(x,y) = \inf_{\substack{f \in Lip(1)\\ \nabla_{yx}f = 1}} \nabla_{xy} \Delta f, 
\end{equation*}
where $\nabla_{xy}f=\frac{f(x)-f(y)}{d(x,y)}$ and $d$ is graph distance function. 
\end{theorem}

In this note, we first show that a positively curved planar graph $G$ with $\delta(G)\geq 3$ has bounded maximum degree. 

\begin{theorem}\label{thm:bounded-degree}
Let $G$ be a simple positively curved planar graph $G$ with $\delta(G) \geq 3$. Then $\Delta(G) \leq 17$.
\end{theorem}

The bound $\Delta(G)\leq 17$ is almost sharp as there exists a positively curved planar graph on $17$ vertices with minimum degree at least $3$ and maximum degree $16$ (see Figure \ref{fig:extremal16}). We guess that the maximum degree bound in Theorem \ref{thm:bounded-degree} could be improved to $16$.

\vspace*{-1.6cm}

\begin{figure}[htb]
\begin{center}

\tikzstyle{every node} = [circle, draw, fill=black, inner sep = 0.5pt, minimum width= 1pt]

\begin{tikzpicture}[scale=2]

    \pgfmathsetmacro\unitangle{360/16};
    \pgfmathsetmacro\offset{1};
    \pgfmathsetmacro\edge{1};
    \node (z) at (0,0) {};

    \foreach \j in {0,...,16}{
    		\node (y{\j}) at (\j * \unitangle: \offset) {};
    	    \draw (z) --++ (\j * \unitangle:  \edge) node (y{\j}) {};
    }

	\foreach \j in {0,...,16}{
		\pgfmathsetmacro\next{int(Mod(int(\j+1),16))};
		\draw (y{\j}) --++ (y{\next});
				
		\pgfmathsetmacro\jumpb{int(Mod(int(\j+2),16))};
		\pgfmathsetmacro\jumpa{int(Mod(int(\j+1),16))};
		\pgfmathsetmacro\jumpd{int(Mod(int(\j+4),16))};

		\ifthenelse{\intcalcMod{\j}{2}=0}
	    	{\draw (y{\j}) --++ (y{\jumpb}) --cycle;}
	    	{}
	    
	    \ifthenelse{\intcalcMod{\j}{4}=0}
	    	{\draw (y{\j}) --++ (y{\jumpd}) --cycle;}
	    	{}
	}
\end{tikzpicture}
\end{center}

\vspace*{-2.2cm}

\caption{A planar graph on $16$ vertices with positive Ricci curvature}
\label{fig:extremal16}
\end{figure}
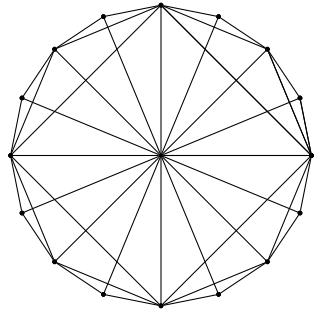

Since a positively curved graph with bounded degree also has bounded diameter by Lemma \ref{lem:diam}. It then follows that analogous to the case of combinatorial curvature, a simple positively (LLY)-curved planar graph has bounded number of vertices. 

\begin{theorem}\label{thm:min-degree3}
If $G$ is a simple positively curved planar graph with minimum degree at least $3$, then $G$ is finite. In particular, $|V(G)| \leq 17^{544}$.
\end{theorem}

A natural interesting question is to improve the upper bound on the number of the vertices of a positively curved planar graph. We also remark that Theorem \ref{thm:min-degree3} is no longer true if we drop the the condition $\delta(G) \geq 3$. Indeed, Figure \ref{fig:infinite} shows two positively curved infinite planar graphs whose minimum degree is $2$. The dashed line denotes an edge that can either be added into the graph or left out. Another interesting question would be to classify all positively curved infinite planar graphs.

\begin{figure}[htb]
	\begin{center}
        \begin{minipage}{.3\textwidth}
        		\resizebox{6cm}{!}{\begin{tikzpicture}[scale=1]
\usetikzlibrary{calc}

\tikzset{main style/.style={circle, draw, fill=black, inner sep = 1pt, minimum width= 1pt}}

    \node[main style] (a) at (0,1) {};
    \node[main style] (b) at (0,-1) {};
    
    \draw[dashed] (a) -- (b);
    
    \foreach \j in {1,...,4}{
    		\node[main style] (x{\j}) at (\j-0.5,0) {};
        	\node[main style] (y{\j}) at (-\j+0.5,0) {};
    	    \draw (a) -- (x{\j});
    	    \draw (a) -- (y{\j});
    	    \draw (b) -- (x{\j});
    	    \draw (b) -- (y{\j}); 
    }

    \path (x{2}) -- node[auto=false]{\ldots} (x{3});
    \path (y{2}) -- node[auto=false]{\ldots} (y{3});

\end{tikzpicture}}
        \end{minipage}
            \hspace{2cm}
        \begin{minipage}{.3\textwidth}
        		\resizebox{4.5cm}{!}{\begin{tikzpicture}[scale=1]

\tikzset{main style/.style={circle, draw, fill=black, inner sep = 0.8pt, minimum width= 0.8pt}}

    \pgfmathsetmacro\offset{0.5};

    \node[main style] (a) at (90:  \offset) {};
    \node[main style] (b) at (210: \offset) {};
    \node[main style] (c) at (-30: \offset) {};
    
    \draw (b) -- (a) -- (c);
    \draw[dashed] (c) -- (b);

    \foreach \j in {1,...,4}{
    		\node[main style] (x{\j}) at (30: \j * \offset) {};
        	\node[main style] (y{\j}) at (145:\j * \offset) {};
        	\node[main style] (z{\j}) at (-90:\j * \offset) {};

    	    \draw (a) -- (x{\j});
    	    \draw (c) -- (x{\j});
    	    \draw (a) -- (y{\j});
    	    \draw (b) -- (y{\j}); 
    	    \draw (b) -- (z{\j});
    	    \draw (c) -- (z{\j}); 
    }

\end{tikzpicture}}
        \end{minipage}
    \end{center} 
    \caption{Positively curved infinite planar graphs}
    \label{fig:infinite}
\end{figure}
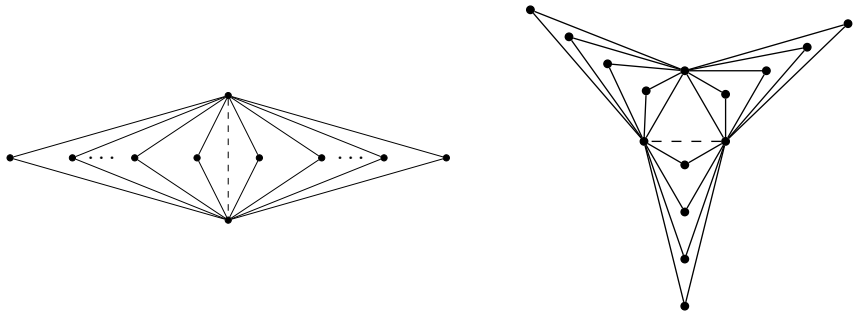

\textbf{Notation:} For convenience, in the remaining of the paper, all subsequent appearances of $\kappa$ are referring to the Lin-Lu-Yau Ricci curvature. We also use $\gamma(x,y)$ to denote $|\Gamma(x,y)|$, i.e., the number of common neighbors of $x,y$. Moreover, let $[a,b] = \{ i\in \mathbb{Z}: a\leq i\leq b\}$.

\begin{section}{Proof of Theorem \ref{thm:bounded-degree}}

\begin{lemma}\label{lem:y_prop}
Let $G$ be a simple positively curved planar graph and $xy$ be an edge with $d(x) \geq d(y)$. Suppose $S \subseteq \Gamma(y)\backslash \{x\}$, $|S| = s$, and $|S \cap \Gamma(x)| = k$.  Then 
$$|\Gamma(S) \cap \Gamma(x)| > \frac{s}{d(y)}d(x) - (k+1+|\Gamma(x,y)|) + |\Gamma(S)\cap \Gamma(x,y)|.$$

\end{lemma}
\begin{proof}[Proof of Lemma \ref{lem:y_prop}]
Suppose not. Define a 1-Lipschitz function $f$ as follows:

\[ f(u) =   \begin{cases} 
                       -1 & u \in \Gamma(x)\backslash (\Gamma(S)\cup \Gamma(y)), \\
                        1 & \textrm{$u=y$ or $u\in S$,}\\
                        0 & \textrm{otherwise.}
            \end{cases}   
\]
  Then by Theorem \ref{thm:curvature_laplacian}, 
\begin{align*}
    \kappa(x,y) & \leq \nabla_{xy} \Delta f \\ 
                & \leq \Delta f(x) - \Delta f(y) \\
                & \leq \frac{k+1 - \lp d(x)-|\Gamma(x,y)|-|\Gamma(S)\cap \Gamma(x)| + |\Gamma(S)\cap \Gamma(x,y)|\rp }{d(x)} -\lp -\frac{d(y)-s}{d(y)} \rp\\
                & \leq 0, 
\end{align*}
which contradicts that $G$ is a positively curved graph.
\end{proof}



\begin{proof}[Proof of Theorem \ref{thm:bounded-degree}]

Let $t$ be the maximum degree and
$x$ be a max-degree vertex in $G$ and $\Gamma(x) = \{v_0, v_1, \cdots, v_{t-1}\}$ be the neighbors of $x$ listed in the clockwise order of a planar drawing of $G$. Assume for the sake of contradiction that $d(x) \geq 18$. For convenience, assume $v_{\ell} \equiv v_{\ell \textrm{ mod } t}$. Given two vertices $v_k$ and $v_{k+\ell}$, if $v_k v_{k+\ell} \in E(G)$, we say $v_k v_{k+\ell}$ forms an \textit{$\ell$-arc}; we say $v_k v_{k+\ell}$ forms an $\ell$-cap if either $v_k v_{k+\ell}$ forms an $\ell$-arc or there exists some vertex $z \in V(G)\backslash \{x\}$ such that $v_k z , v_{k+\ell} z \in E(G)$ but $v_{k+i} z \notin E(G)$ for all $i\in [1,\ell-1]$.


\begin{claim}\label{cl:no_close_triangle}
Suppose $v_k v_{k+\ell}$ forms an $\ell$-cap for some $\ell \geq 3$. Then there exists some $i \in \{k, k+1, \cdots, k+ \ell-2\}$ such that $v_i v_{i+2}$ forms a $2$-cap. In addition, $\Gamma(v_{i+1}, x) =\{v_i, v_{i+2}\}$. 
\end{claim}
\begin{proof}[Proof of Claim \ref{cl:no_close_triangle}]
WLOG assume that $k, k+\ell \in [0,t-1]$. Let $i,j$ be picked (among all pairs in $[k,k+\ell]$) such that $v_i v_j$ forms a $(j-i)$-cap, $j-i$ is at least $2$ and minimal.

We first show that $j-i\leq 3$. If $j-i\geq 4$, then by the minimality of $\{i,j\}$, $\gamma(x,v_{i+2}) = |\Gamma(x, v_{i+2})| \leq 2$. In particular, $\Gamma(v_{i+2,x}) \subseteq \{v_{i+1}, v_{i+3}\}$.
Consider $S = \Gamma(v_{i+2})\backslash \{x\}$. By Lemma \ref{lem:y_prop}, 
\begin{align*}
    |\Gamma(S)\cap \Gamma(x)| &> \frac{|S|}{d(v_{i+2})} d(x) -(2\gamma(x,v_{i+2})+1) + |\Gamma(x, v_{i+2}) \cap \Gamma(S)| \\
                             & \geq \frac{2}{3}d(x) -5.
\end{align*}
It follows that $|\Gamma(S)\cap \Gamma(x)| \geq 8$ since $d(x) \geq 18$. Again, by the minimality of $j-i$, there must exist a vertex $u \in \Gamma(v_{i+2})\backslash \Gamma(x)$ such that $u$ is adjacent to $v_{t}$ for some $t\geq i+7$. Since $v_{i+2} v_t$ do not form a $2$-cap by the minimality of $j-i$, it follows that $u$ is adjacent to every vertex in $\{v_{i+2}, v_{i+3}, \cdots, v_{i+6}\}$. Now applying Lemma \ref{lem:y_prop} on the edge $x v_{i+4}$ by setting $S' = \Gamma(v_{i+4}) \backslash \{x,u\}$, we obtain that 
\begin{align*}
    |\Gamma(S')\cap \Gamma(x)| &> \frac{|S'|}{d(v_{i+4})} d(x) -(2\gamma(x,v_{i+4})+1) + |\Gamma(x, v_{i+4}) \cap \Gamma(S')| \\
                               & \geq \min\{\frac{1}{3}d(x)-3, \frac{2}{4}d(x)-5\},
\end{align*}
which implies that  $|\Gamma(S')\cap \Gamma(x)|\geq 4$ since $d(x)\geq 18$. However, observe that $\Gamma(S')\cap \Gamma(x)\subseteq \{v_{i+2}, v_{i+4}, v_{i+6}\}$, which gives us a contradiction. Hence by contradiction $|j-i| \leq 3$.




Next we show that $j-i=2$. Suppose otherwise that $j-i = 3$. Let $S = \Gamma(v_{i+1})\backslash\{x, v_i\}$. 
If $d(v_{i+1}) \geq 4$, then by Lemma \ref{lem:y_prop} we have 
$$\frac{2}{4}d(x)- 4 < |\Gamma(S)\cap \Gamma(x)| \leq 4,$$ which gives a contradiction since $d(x) \geq 18$. Otherwise $d(v_{i+1}) = 3$. 
Now if $S = \{v_{i+2}\}$, then by the minimality of $\{i,j\}$,
$$ \frac{1}{3} d(x) -4 < |\Gamma(S)\cap \Gamma(x)| \leq 2, $$
which gives a contradiction since $d(x) \geq 18$. Otherwise $S = \{u\}$, for some $u \notin \Gamma(x)$. Then Lemma \ref{lem:y_prop} implies that
$$|\Gamma(S)\cap \Gamma(x)| > \frac{1}{3} d(x) -2 \geq 4,$$
i.e., $|\Gamma(S)\cap \Gamma(x)| \geq 5$ since $d(x)\geq 18$. However, observe that since $\{v_i, v_j\}$ forms a $2$-cap, $|\Gamma(S)\cap \Gamma(x)| \leq j-i + 1 = 4$, which gives us a contradiction.

Hence we have $j-i = 2$. Similar reasoning gives that $\Gamma(v_{i+1}) = \{x, v_i, v_j\}$. This completes the proof of Claim \ref{cl:no_close_triangle}.

\end{proof}


\begin{claim}\label{cl:no_internal_2}
Suppose $v_k v_{k+\ell}$ forms an $\ell$-cap for some $4 \leq \ell \leq 8$. Then there does not exist $i,j \in [k+1, k+\ell-1]$ such that $v_i v_j$ forms a $(j-i)$-cap and $j-i\geq 2$.
\end{claim}
\begin{proof}[Proof of Claim \ref{cl:no_internal_2}]
Suppose for the sake of contradiction that there exist $i,j\in [k+1, k+\ell-1]$ such that $v_i v_j$ is a cap and $j-i\geq 2$. By Claim \ref{cl:no_close_triangle}, there exists some $i',j' \in [i,j]$ such that $v_i'v_j'$ forms a $2$-cap and $\Gamma(v_{i'+1}) = \{v_{i'},v_{j'},x\}$. Consider the edge $x v_{i'+1}$ and apply Lemma \ref{lem:y_prop} with $S = \Gamma(v_{i'+1})\backslash \{x\}$. It follows that 
$$|\Gamma(S)\cap \Gamma(x)| > \frac{|S|}{d(v_{i'+1})} d(x) - 5 + |\Gamma(S)\cap \Gamma(x, v_{i'+1})|,$$
which implies that $|\Gamma(S)\cap \Gamma(x)| \geq 10$ if $v_{i'} v_{j'}\in E(G)$ and $|\Gamma(S)\cap \Gamma(x)| \geq 8$ otherwise. In both cases, it contradicts that $\ell \leq 8$ since $\Gamma(S)\cap \Gamma(x) \subseteq \{v_s: s\in [k,k+\ell]\}$.
\end{proof}

\begin{claim}\label{cl:no_large_closure}
There exist no $i,j$ such that $v_i v_j$ forms a $\ell$-cap where $\ell \notin \{1,2,4\}$.
\end{claim}
\begin{proof}[Proof of Claim \ref{cl:no_large_closure}]
We first show that there does not exist a $3$-cap. Suppose on the contrary that there is some $3$-cap $v_i v_j$ such that $i,j \in [0,t-1]$ and $j-i =3$. By Claim \ref{cl:no_close_triangle}, there exists some $i',j' \in [i,j]$ such that $v_{i'}v_{j'}$ is a $2$-cap. WLOG, suppose $v_i v_{i+2}$ is a $2$-cap and it follows from Claim \ref{cl:no_close_triangle} that $\Gamma(v_{i+1}) = \{x, v_i, v_{i+2}\}$. Now consider the edge $x v_{i+1}$. Applying lemma \ref{lem:y_prop} with $S = \{v_{i+2}\}$, we have that 
$$|\Gamma(v_{i+2})\cap \Gamma(x)|= |\Gamma(S)\cap \Gamma(x)| > \frac{1}{d(v_{i+1})} d(x)- 4 + |\Gamma(x, v_{i+1}) \cap \Gamma(v_{i+2})|,$$
which implies that $|\Gamma(v_{i+2})\cap \Gamma(x)|\geq 3 + |\Gamma(x, v_{i+1}) \cap \Gamma(v_{i+2})|$ since $d(x) \geq 18$. This contradicts that $|j-i| = 3$ since $\Gamma(v_{i+2})\cap \Gamma(x) \subseteq \{v_i, v_{i+1}, v_{i+3}\}$.

Now we will show by (strong) induction that there exist no $i,j$ such that $v_i v_j$ is a $\ell$-cap for $\ell \notin \{1,2,4\}$. Again WLOG we assume that $i,j\in [0,t-1]$ and $i<j$. For the base case, suppose $j-i=5$. By Claim \ref{cl:no_close_triangle}, there exists some $i',j'\in [i,j]$ such that $v_{i'} v_{j'}$ is a $2$-cap. By Claim \ref{cl:no_internal_2}, $\{i,j\}\cap \{i',j'\} \neq \emptyset$. WLOG, assume that $i' = i$ and thus $j' = i+2$. Similar to before, 
$$|\Gamma(v_{i+2})\cap \Gamma(x)| \geq 3 + |\Gamma(x, v_{i+1}) \cap \Gamma(v_{i+2})|.$$
Hence $v_{i+2}$ is adjacent to either $v_{i+4}$ or $v_{i+5}$. The former case contradicts Claim \ref{cl:no_internal_2} and the later case contradicts that there is no $3$-cap. Hence by contradiction, we obtain that there exists no $5$-cap.

Now assume that there exists no $s$-cap for all $s \in [\ell-1]\backslash \{1,2,4\}$ for $\ell \geq 6$. We want to show that there exists no $\ell$-cap. Suppose on the contrary that there exists some $k$ such that $v_k v_{k+\ell}$ forms an $\ell$-cap. WLOG, assume that $v_k, v_{k+\ell} \in [0,t-1]$. 

We now claim that there exists some $i \in \{k+1, k+2, \cdots, k+\ell-3\}$ such that $v_i v_{i+2}$ forms a $2$-cap. By Claim \ref{cl:no_close_triangle}, there must exist a $2$-cap $v_i v_{i+2}$ within the $\ell$-cap $v_k v_{k+\ell}$. If $i \in [k+1, k+\ell-3]$, we are done. Otherwise assume that $i = k$ and it follows that $\Gamma(v_{k+1}) = \{x, v_k, v_{k+2}\}$. Applying Lemma \ref{lem:y_prop} on the edge $xv_{k+1}$ with $S = \{v_{k+2}\}$, we obtain that 

$$|\Gamma(v_{k+2},x)|  > \frac{1}{3} d(x) -4 + |\Gamma(x, v_{k+1}) \cap \Gamma(v_{k+2})|.$$
Since $d(x) \geq 18$, it follows that $v_{k+2}$ has at least $2$ neighbors $v_j \in \Gamma(x)$ with $j > k+2$. This then implies that $v_{k+2}$ has at least one neighbor $v_j$ with $j \geq k+4$. By the induction hypothesis, there exists no $s$-cap for $s \in [\ell-1] \backslash \{1,2,4\}$. It follows that either $j = k+4$ or $j = k+6$. If $j=k+4$, then $v_{k+2} v_{k+4}$ forms a $2$-cap and we are done. Otherwise $v_{k+2} v_{k+6}$ forms a $4$-cap. Now Claim \ref{cl:no_close_triangle} and Claim \ref{cl:no_internal_2} imply that there must exist a $2$-cap within the $4$-cap $v_{k+2} v_{k+6}$ that is either $v_{k+2} v_{k+4}$ or $v_{k+4} v_{k+6}$. If $v_{k+2} v_{k+4}$ is a cap, we are done. Otherwise $v_{k+4} v_{k+6}$ is a $2$-cap and $\Gamma(v_{k+5},x) = \{v_{k+4}, v_{k+6}\}$. Similar to before, applying Lemma \ref{lem:y_prop} on the edge $x v_{k+5}$ with $S = \{v_{k+4}\}$, we obtain that  $v_{k+4}$ must have two neighbors $v_j \in \Gamma(x)$ such that $j < k+4$. Since $v_{k+2} v_{k+6}$ forms a cap, it must follows that $v_{k+2} v_{k+4} \in E(G)$ and forms a $2$-cap. Thus we showed that there exists some $i \in [k+1, k+\ell-3]$ such that $v_i v_{i+2}$ forms a $2$-cap. Moreover, $\Gamma(v_{i+1},x) = \{v_i, v_{i+2}\}$.

Now applying Lemma \ref{lem:y_prop} on the edge $xv_{i+1}$ by setting $S = \{v_i, v_{i+2}\}$, we obtain that 
$$|\Gamma(S)\cap \Gamma(x)| > \frac{2}{3}d(x)-5+ |\Gamma(x,v_{i+1}) \cap \Gamma(S)|.$$

If $v_i v_{i+2} \in E(G)$, then it follows that $|\Gamma(S)\cap \Gamma(x)| > \frac{2}{3} \cdot 18 -5 + 2 \geq 10$. However, since there are no $s$-caps for any $s\in [\ell]\backslash\{1,2,4\}$, $|\lp\Gamma(v_i) \cup \Gamma(v_{i+2})\rp \cap \Gamma(x)| \leq 9$, which gives us a contradiction. On the other hand, if $v_i v_{i+2} \notin E(G)$, then it follows that $|\Gamma(S)\cap \Gamma(x)| > \frac{2}{3} \cdot 18 -5  \geq 8$. Similarly, since there are no $s$-caps for any $s\in [\ell]\backslash\{1,2,4\}$, $|\lp\Gamma(v_i) \cup \Gamma(v_{i+2})\rp \cap \Gamma(x)| \leq 7$, which again gives a contradiction. Therefore, by contradiction, there exists no $\ell$-cap. Claim \ref{cl:no_large_closure} then follows from induction.
\end{proof}

Now we are ready to complete the proof of Theorem \ref{thm:bounded-degree}. We first claim that $\Gamma(x)$ must be an independent set. Suppose otherwise that $v_{i} v_j \in E(G)$. Then observe that $v_i v_j$ forms both a $(j-i)$-cap and $(d(x)-(j-i))$-cap. At least one of them is an $\ell$-cap with $\ell \geq 8$, which does not exist by Claim \ref{cl:no_large_closure} since $d(x) \geq 18$. Hence $\Gamma(x)$ is an independent set.
We now claim that there must be a $2$-cap in the neighborhood of $x$. Suppose not. Then by Claim \ref{cl:no_close_triangle}, there does not exist any $\ell$-cap for any $\ell \geq 2$. Applying Lemma \ref{lem:y_prop} on the edge $x v_{0}$ with $S = \Gamma(v_0) \backslash \{x\}$, it follows that 

$$|\Gamma(S) \cap \Gamma(x)| > \frac{2}{3} d(x) - 1,$$
i.e., $|\Gamma(S) \cap \Gamma(x)| \geq 12$ since $d(x) \geq 18$.
Hence WLOG there exists some vertex $u \in \Gamma(v_0) \backslash \Gamma(x)$, such that $uv_j\in E(G)$ for some $j\geq 5$. Since there is no $\ell$-cap for any $\ell \geq 2$, it follows that $u v_i \in E(G)$ for all $i \in [0,5]$. Applying Lemma \ref{lem:y_prop} on the edge $x v_2$ with $S = \Gamma(v_2) \backslash \{u,x\}$, we have 

$$|\Gamma(S) \cap \Gamma(x)| > \frac{1}{3} d(x) -1 \geq 5,$$
which gives a contradiction since $\Gamma(S) \cap \Gamma(x) \subseteq \{v_1, v_2, v_3\}$. Hence there exists some $2$-cap $v_i v_{i+2}$ and $\Gamma(v_{i+1},x) = \{v_i, v_{i+2}\}$, which contradicts that $\Gamma(x)$ is independent. By contradiction, it follows that $\Delta(G) \leq 17$.

\end{proof}

\section{Proof of Theorem \ref{thm:min-degree3}}

Bourne et.al. \cite{BCLMP18} showed an ``integer-valuedness'' property of the optimal Kantorovich potential as follows:

\begin{lemma}\cite{BCLMP18}\label{lem:integer-valued}
Let $G = (V,E)$ be a locally finite graph. Let $x,y \in V$ with $x\sim y$. Let $p\in [0,1]$. Then there exists a $1$-lipschitz function $\phi$ such that 
$$W(m_x^p, m_y^p) = \dss_{w\in V} \phi(w)\lp m_x^p(w) - m_y^p(w) \rp,$$
and $\phi(w) \in \mathbb{Z}$ for all $w\in V$.
\end{lemma}

Before we prove Theorem \ref{thm:min-degree3}, we show a simple lemma that gives a lower bound on the Ricci curvature of a positively curved graph with bounded degree.

\begin{lemma}\label{lem:curvature_max_degree}
Suppose $G$ is a positively curved graph with maximum degree $\Delta$. Then $\kappa(G) \geq \frac{1}{\Delta(\Delta-1)}$. 
\end{lemma}
\begin{proof}
It was shown in \cite{BCLMP18} that for every $xy \in E(G)$,
$$\kappa(x,y) \geq \kappa_0(x,y) = 1 - W(m_x^0, m_y^0).$$
By Lemma \ref{lem:integer-valued}, there exists some integer valued $1$-Lipschitz function $\phi$ such that 
$$W(m_x^0, m_y^0) = \dss_{w\in V} \phi(w)\lp m_x^0(w) - m_y^0(w) \rp,$$
Observe that by definition of $m_x(\cdot)$, and $m_y(\cdot)$, 
$$|m_x^0(w) -m_y^0(w)| = \frac{c}{d(x)d(y)},$$
for some non-negative integer $c$.
Since $G$ is positive curved and $\phi(\cdot)$ is an integer-valued function, it then follows that $$\kappa(x,y) = 1 - W(m_x^0, m_y^0) \geq \frac{1}{d(x)d(y)} \geq \frac{1}{\Delta(\Delta-1)}.$$
\end{proof}

\begin{proof}[Proof of Theorem \ref{thm:min-degree3}]

Let $G$ be a positively curved planar graph with $\delta(G) \geq 3$. 
By Theorem \ref{thm:bounded-degree}, we obtain that $\Delta(G) \leq 17$. It then follows from Lemma \ref{lem:curvature_max_degree} that $\kappa_{LLY}(G)\geq \frac{1}{17 \cdot 16} = \frac{1}{272}$. By Lemma \ref{lem:diam}, $diam(G) \leq \frac{2}{\kappa_{LLY}(G)} \leq 544$. It then follows trivially that 
$$|V(G)|\leq 1+\sum_{i=1}^{diam(G)-1} \Delta(G)\cdot (\Delta(G)-1)^{i} <17^{544}.$$

\end{proof}
\end{section}

\end{document}